\documentclass[12pt, twoside, leqno]{article}

\usepackage{amsmath,amsthm}
\usepackage{amssymb}

\usepackage{enumitem}

\usepackage{graphicx}

\usepackage[T1]{fontenc}
\usepackage{enumerate, latexsym, amsmath, amsfonts, amssymb, amsthm, graphicx, color}
\usepackage{bm}
\def\pmod #1{\ ({\rm{mod}}\ #1)}

\def\F{\Bbb F}

\def\bg{\bigg}
\def\({\bg(}
\def\){\bg)}

\def\mod #1{\ {\rm mod}\ #1}

\def\s{{\bm s}}
\def\O{{\mathcal{O}}}
\def\p{{\mathfrak{p}}}

\def\sign{{\rm sign}}
\def\Gal{{\rm Gal}}

\newtheorem{theorem}{Theorem}[section]
\newtheorem{corollary}[theorem]{Corollary}
\newtheorem{lemma}[theorem]{Lemma}

\theoremstyle{definition}

\newtheorem{remark}[theorem]{Remark}

\numberwithin{equation}{section}

\begin{document}


\baselineskip=17pt


\title{On finite field analogues of determinants involving the Beta function}

\author{Hai-Liang Wu\\
	School of Science\\ 
	Nanjing University of Posts and Telecommunications\\
	210023 Nanjing, People's Republic of China\\
	E-mail: whl.math@smail.nju.edu.cn
	\and 
	Li-Yuan Wang\\
	School of Physical and Mathematical Sciences\\ 
	Nanjing Tech University\\
	211816 Nanjing, People's Republic of China\\
	E-mail: wly@smail.nju.edu.cn
	\and 
	Hao Pan*\\
	School of Applied Mathematics\\ 
	Nanjing University of Finance and Economics\\
	210046 Nanjing, People's Republic of China\\
	E-mail: haopan79@zoho.com}
\date{}

\maketitle


\renewcommand{\thefootnote}{}


\footnote{*Corresponding author.}
\renewcommand{\thefootnote}{\arabic{footnote}}
\setcounter{footnote}{0}


\begin{abstract}
Motivated by the works of L. Carlitz, R. Chapman and Z.-W. Sun on cyclotomic matrices, in this paper, we investigate certain cyclotomic matrices concerning Jacobi sums over finite fields, which can be viewed as finite field analogues of certain matrices involving the Beta function. For example, let $q>1$ be a prime power and let $\chi$ be a generator of the group of all multiplicative characters of $\mathbb{F}_q$. Then we prove that 
$$\det\left[J_q(\chi^i,\chi^j)\right]_{1\le i,j\le q-2}=(q-1)^{q-3},$$
where $J_q(\chi^i,\chi^j)$ is the Jacobi sum over $\mathbb{F}_q$. This is a finite analogue of 
$$\det [B(i,j)]_{1\le i,j\le n}=(-1)^{\frac{n(n-1)}{2}}\prod_{r=0}^{n-1}\frac{(r!)^3}{(n+r)!},$$
where $B$ is the Beta function. Also, if $q=p\ge5$ is an odd prime, then we show that 
$$\det \left[J_p(\chi^{2i},\chi^{2j})\right]_{1\le i,j\le (p-3)/2}=\frac{1+(-1)^{\frac{p+1}{2}}p}{4}\left(\frac{p-1}{2}\right)^{\frac{p-5}{2}}.$$
\end{abstract}

\section{Introduction}
\subsection{Notations}
Let $q=p^n$ be a prime power with $p$ prime and $n\in\mathbb{Z}^+$. Let $\mathbb{F}_q$ be the finite field of $q$ elements, and let $\mathbb{F}_q^{\times}$ be the group of all nonzero elements of $\mathbb{F}_q$. Let $\widehat{\mathbb{F}_q^{\times}}$ denote the group of all multiplicative characters of $\mathbb{F}_q$. In addition, for any $\psi\in\widehat{\mathbb{F}_q^{\times}}$ we define $\psi(0)=0$. Throughout out this paper, we let $\chi$ be a generator of $\widehat{\mathbb{F}_q^{\times}}$. Also, For any $\chi^i,\chi^j\in\widehat{\mathbb{F}_q^{\times}}$, the Jacobi sum of $\chi^i$ and $\chi^j$ is defined by 
\begin{equation*}
	J_q(\chi^i,\chi^j):=\sum_{x\in\mathbb{F}_q}\chi^i(x)\chi^j(1-x).
\end{equation*}

\subsection{Background and Motivations} We first introduce some related works involving cyclotomic matrices. The earliest research on cyclotomic matrices came from Lehmer \cite{Lehmer} and Carlitz \cite{carlitz}. For example, 
given a nontrivial character $\psi\in\widehat{\mathbb{F}_p^{\times}}$, 
Carlitz studied the characteristic polynomial of the cyclotomic matrix 
$$C_p(\psi):=\left[\psi(j-i)\right]_{1\le i,j\le p-1}.$$ In particular, when $\psi$ is the quadratic character, Carlitz showed that the characteristic polynomial of $C_p(\psi)$ is 
$$\left(t^2-(-1)^{\frac{p-1}{2}}p\right)^{\frac{p-3}{2}}\left(t^2-(-1)^{\frac{p-1}{2}}\right).$$ 

Along this line, Chapman \cite{chapman} initiated the study of several variants of Carlitz's results. Surprisingly, these variants have close connections with the real quadratic field $\mathbb{Q}(\sqrt{p})$. For example, let $(\frac{\cdot}{p})$ be the Legendre symbol and let $$\varepsilon_p^{(2-(\frac{2}{p}))h_p}=a_p+b_p\sqrt{p}\ (a_p,b_p\in\mathbb{Q}),$$
where $\varepsilon_p>1$ and $h_p$ are the fundamental unit and class number
of $\mathbb{Q}(\sqrt{p})$ respectively. Chapman conjectured that 
$$\det\left[\left(\frac{j-i}{p}\right)\right]_{1\le i,j\le \frac{p+1}{2}}=\begin{cases}
	-a_p & \mbox{if}\ p\equiv 1\pmod4,\\
	1    & \mbox{if}\ p\equiv 3\pmod4.
\end{cases}$$
This challenging conjecture was later known as Chapman's ``evil determinant" conjecture, and was confirmed by Vsemirnov \cite{V1,V2}. 

In recent years, Sun \cite{S19} and Krachun and his collaborators \cite{Krachun} studied some  variants of the above determinant. For example, Sun proved that $$-\det\left[\left(\frac{i^2+j^2}{p}\right)\right]_{1\le i,j\le (p-1)/2}$$
is a quadratic residue modulo $p$. Later the first author \cite{W21} proved that for any $c\in\F_p$ with $c\neq\pm2$, 
$$\det\left[\left(\frac{i^2+cij+j^2}{p}\right)\right]_{0\le i,j\le p-1}=0$$
whenever the curve defined by the equation $y^2=x(x^2+cx+1)$ is a supersingular elliptic curve over $\F_p$. Readers may refer to \cite{Krachun,S19,W21,WSW} for recent progress on this topic. 

After introducing the above relevant research results, we now describe our research motivations. 

For any complex numbers $x,y$ with Re$(x)>0$ and Re$(y)>0$, the Beta function is defined by 
$$B(x,y)=\int_{0}^{1}t^{x-1}(1-t)^{y-1}dt.$$
If we define a function 
$$J_q:\ \widehat{\mathbb{F}_q^{\times}}\times\widehat{\mathbb{F}_q^{\times}}\rightarrow\mathbb{Q}(e^{2\pi{\bf i}/(q-1)})$$
by sending $(\chi^i,\chi^j)$ to the Jacobi sum $J_q(\chi^i,\chi^j)$, then it is known that $J_q$ is indeed a finite analogue of the Beta function. Readers may refer to \cite{F,Greene} for the detailed introduction on this topic.

Let $n$ be a positive integer. We first consider the following matrix
$$B_n:=\left[B(i,j)\right]_{1\le i,j\le n}.$$
It is known that for any positive integers $i,j$ we have 
$$B(i,j)=\frac{\Gamma(i)\Gamma(j)}{\Gamma(i+j)},$$
where $\Gamma(\cdot)$ is the Gamma function. By this and \cite[(4.8)]{N} we obtain 
\begin{equation}\label{Eq. det of Bn}
	\det B_n=(-1)^{\frac{n(n-1)}{2}}\prod_{r=0}^{n-1}\frac{(r!)^3}{(n+r)!}.
\end{equation}
Recall that $\chi$ is a generator of $\widehat{\mathbb{F}_q^{\times}}$. Since $J_q$ is a finite field analogue of the Beta function and $\mathbb{F}_q^{\times}\cong\widehat{\mathbb{F}_q^{\times}}$, the determinant 
\begin{equation}\label{Eq. finite field analogue}
	\det\left[J_q(\chi^i,\chi^j)\right]_{1\le i,j\le q-2}
\end{equation}
can be naturally viewed as a finite field analogue of (\ref{Eq. det of Bn}). 

Motivated by the above results, it is natural to consider the matrices with Jacobi sums as their entries. In general, for any positive integer $k\mid q-1$, we define the cyclotomic matrix
\begin{equation}\label{Eq. definition of Jq(k)}
	J_{\chi,q}(k):=\left[J_q(\chi^{ki},\chi^{kj})\right]_{1\le i,j\le (q-1-k)/k}
\end{equation}
which concerns the Jacobi sums of nontrivial $k$-th multiplicative characters. For some large $k$, it is easy to evaluate $\det J_{\chi,q}(k)$. For example, 
\begin{itemize}
	\item $k=(q-1)/2$: in this case we have 
	$$\det J_{\chi,q}(k)=J_q(\chi_2,\chi_2)=(-1)^{(q+1)/2},$$
	where $\chi_2=\chi^{(q-1)/2}$ is the quadratic character.
	
	\item  $k=(q-1)/3$: in this case we have $q\equiv 1\pmod 3$ and 
	$$\det J_{\chi,q}(k)=\left|\begin{array}{cc}
		J_q(\chi_3,\chi_3)         & J_q(\chi_3,\bar{\chi}_3)\\
		J_q(\chi_3,\bar{\chi}_3)   & J_q(\bar{\chi_3},\bar{\chi}_3)
	\end{array}\right|=\left|\begin{array}{cc}
		J_q(\chi_3,\chi_3)         & -1\\
		-1   & J_q(\bar{\chi_3},\bar{\chi}_3)
	\end{array}\right|=q-1,$$
	where $\chi_3=\chi^{(q-1)/3}$ is a character of order $3$. 
\end{itemize}

\subsection{Some Geometric Interpretations} Here we briefly introduce some geometric interpretations of (\ref{Eq. definition of Jq(k)}). Readers may refer to \cite[Section 6]{F} for a detailed introduction on this topic.

Let notations be as above and let $m=(q-1)/k$. For any $1\le i,j\le m-1$, consider the Fermat curves 
$$y^m=x^i(1-x)^j.$$ Let $\chi_m$ be a character of order $m$. In such cases, their Galois representations can be determined by the Jacobi sums $J_q(\chi_m^{si},\chi_m^{sj})$, and their period matrix can be expressed in terms of the beta values $B\left(\frac{si}{m},\frac{sj}{m}\right)$ for certain $s\in\mathbb{Z}/m\mathbb{Z}$. 

By the above, if we let $\chi=\chi_{\p}$ be the Teich\"{u}muller character defined by (\ref{Eq. definition of the T character}) in Section 2 and let $\chi_m=\chi_{\p}^{-k}$ be a character of order $m$, then the matrix $$\left[J_q(\chi_m^{i},\chi_m^{j})\right]_{1\le i,j\le m-1}$$
is a finite field version of the period matrix given by $B\left(\frac{si}{m},\frac{sj}{m}\right)$'s as above.

\subsection{Main Results} Now we state our main results of this paper. 

\begin{theorem}\label{Thm. general Jq(k)}
	Let $q=p^n$ be a prime power with $p$ prime and $n\in\mathbb{Z}^+$. Let $k\mid q-1$ be a positive integer and let $m=(q-1)/k$. Then the following results hold.
	
	{\rm (i)} $\det J_{\chi,q}(k)\in\mathbb{Z}$ and is independent of the choice of the generator $\chi.$
	
	{\rm (ii)} We have the congruence 
	\begin{equation}\label{Eq. congruence of det Jq(k)}
		\det J_{\chi,q}(k)\equiv (-1)^{\frac{(k+1)(m^2-m)}{2}} \pmod p.
	\end{equation}
\end{theorem}

\begin{remark}\label{Remark of Thm. 1}
	As mentioned in Theorem \ref{Thm. general Jq(k)}(i), $\det J_{\chi,q}(k)$ is independent of the choice of the generator $\chi$. Thus, from now on, we abbreviate $\det J_{\chi,q}(k)$ as $\det J_q(k)$. 
\end{remark}

By Theorem \ref{Thm. general Jq(k)} and the congruence (\ref{Eq. congruence involving Jacobi sums}) in Section 2, we can obtain the following result directly. 

\begin{corollary}\label{Corollary of Thm. 1}
	Let $q=p^n$ be a prime power with $p$ prime and $n\in\mathbb{Z}^+$. Let $k\mid q-1$ be a positive integer and let $m=(q-1)/k$. Then 
	$$\det\left[\binom{ki+kj}{ki}\right]_{1\le i,j\le m-1}\equiv (-1)^{\frac{(k+1)m^2-(k-1)m-2}{2}}\pmod p.$$
\end{corollary}

Our next result concerns the explicit value of $\det J_q(1)$.

\begin{theorem}\label{Thm. Jq(1)}
	Let $q=p^n\ge3$ with $p$ prime and $n\in\mathbb{Z}^+$. Then 
	$$\det J_q(1)=(q-1)^{q-3}.$$
\end{theorem}

The next theorem determines the explicit value of $J_p(2)$.

\begin{theorem}\label{Thm. Jq(2)}
	Let $p\ge5$ be an odd prime. Then 
	$$ \det J_p(2)=\frac{1+(-1)^{\frac{p+1}{2}}p}{4}\left(\frac{p-1}{2}\right)^{\frac{p-5}{2}}.$$
\end{theorem}

\begin{remark}
	Unlike Theorem \ref{Thm. Jq(1)}, it is difficult to generalize this result to the finite field $\mathbb{F}_{p^n}$ with $n\ge2$. In Remark \ref{Remark. Why we cannot generalize Thm. 3}, we will provide a detailed explanation of the computational obstacles.
\end{remark}

\subsection{Outline of This Paper} We will prove our main results in Sections 2--4 respectively. In Section 5, we will pose some open problems for further research. 

\section{Proof of Theorem \ref{Thm. general Jq(k)}}

Throughout this section, we let $m=(q-1)/k$. For an integer $a$ with $\gcd(a,m)=1$, the map $x\mod m\mapsto ax\mod m$ is a permutation $\tau_m(a)$ of $\mathbb{Z}/m\mathbb{Z}$. We begin with the following result due to Lerch \cite{Lerch}.

\begin{lemma}\label{Lem. the Lerch permutation}
	Let $\sign({\tau_m(a)})$ denote the sign of $\tau_m(a)$. Then 
	\begin{equation*}
		\sign({\tau_m(a)})=
		\begin{cases}
			(\frac{a}{m})  & \mbox{if}\ m\equiv 1\pmod 2,\\
			1              & \mbox{if}\ m\equiv 2\pmod 4,\\
			(-1)^{(a-1)/2} & \mbox{if}\ m\equiv 0\pmod 4,
		\end{cases}
	\end{equation*}
	where $(\frac{\cdot}{m})$ is the Jacobi symbol if $m$ is odd. In particular,  
	\begin{equation*}
		\sign({\tau_{m}(-1)})=(-1)^{\frac{(m-1)(m-2)}{2}}.
	\end{equation*}
\end{lemma}

We also need the well-known Lucas congruence.

\begin{lemma}\label{Lem. Lucas congruence}
	Let $a,b,c,d\in\mathbb{N}=\{0,1,2,\cdots\}$ with $0\le b,d\le p-1$. Then 
	$$\binom{ap+b}{cp+d}\equiv \binom{a}{c}\binom{b}{d}\pmod p.$$
\end{lemma}

On the other hand, we turn to the following matrix involving the binomial coefficients:
\begin{equation*}
	C_n:=\left[\binom{i+j}{i}\right]_{1\le i,j\le n}=\left[\frac{\Gamma(i+j+1)}{\Gamma(i+1)\Gamma(j+1)}\right]_{1\le i,j\le n}.
\end{equation*}
One can easily verify that (we will give a detailed proof in the appendix)
\begin{equation}\label{Eq. det Cn}
	\det C_n=n+1.
\end{equation}

Let $\zeta_{q-1}=e^{2\pi {\bf i}/(q-1)}$ and let $\zeta_p=e^{2\pi {\bf i}/p}$. Consider the cyclotomic field $L=\mathbb{Q}(\zeta_{q-1},\zeta_p)$ and let $\O_L$ be the ring of all algebraic integers over $L$. Let $\p$ be a prime ideal of $\O_L$ with $p\in\p$. Then it is easy to verify that 
$$\O_L/\p\cong\mathbb{F}_q.$$
Let $\chi_{\p}$ be the Teich\"{u}muller character of $\p$, i.e., 
\begin{equation}\label{Eq. definition of the T character}
	\chi_{\p}(x \mod \p)\equiv x\pmod {\p}
\end{equation}
for any $x\in\O_L$. Then for any integer $1\le i,j\le q-2$, the following congruence holds (see \cite[Proposition 3.6.4]{Cohen}).
\begin{equation}\label{Eq. congruence involving Jacobi sums}
	J_q(\chi_{\p}^{-i},\chi_{\p}^{-j})\equiv -\binom{i+j}{i}\equiv -\frac{\Gamma(i+j+1)}{\Gamma(i+1)\Gamma(j+1)}\pmod{\mathfrak{p}}.
\end{equation}
Congruence (\ref{Eq. congruence involving Jacobi sums}) establishes a connection between binomial coefficients and Jacobi sums, and hence (\ref{Eq. finite field analogue}) can be also viewed as a finite field analogue of (\ref{Eq. det Cn}).

Now we are in a position to prove our first result.

{\noindent{\bf Proof of Theorem \ref{Thm. general Jq(k)}.}} (i) It is known that 
$$\Gal\left(\mathbb{Q}(\zeta_{q-1})/\mathbb{Q}\right)\cong\left(\mathbb{Z}/(q-1)\mathbb{Z}\right)^{\times}.$$ 
Hence 
$$\Gal\left(\mathbb{Q}(\zeta_{q-1})/\mathbb{Q}\right)
=\left\{\sigma_r: r\in\mathbb{Z}\ \text{and}\ \gcd(r,q-1)=1\right\},$$
where $\sigma_r(\zeta_{q-1})=\zeta_{q-1}^r$. Now for any $\sigma_r\in\Gal\left(\mathbb{Q}(\zeta_{q-1})/\mathbb{Q}\right)$, by Lemma \ref{Lem. the Lerch permutation} we have 
\begin{align*}
	\sigma_r\left(\det J_{\chi,q}(k)\right)
	&=\det\left[J_q(\chi^{rki},\chi^{rkj})\right]_{1\le i,j\le m-1}\\
	&=\sign(\tau_m(r))^2 \cdot \det J_{\chi,q}(k)\\
	&=\det J_{\chi,q}(k).
\end{align*}
By the Galois theory and noting that $\det J_{\chi,q}(k)$ is an algebraic integer, we see that $\det J_{\chi,q}(k)\in\mathbb{Z}$. Also, for any integer $r$, the character $\chi^r$ is a generator of $\widehat{\mathbb{F}_q^{\times}}$ if and only if $\gcd(r,q-1)=1$. Thus, $\det J_{\chi,q}(k)$ is independent of the choice of the generator $\chi$ is equivalent to $\det J_{\chi,q}(k)\in\mathbb{Z}$. By the above, 
(i) clearly holds.

(ii) Now we adopt the notations of (\ref{Eq. congruence involving Jacobi sums}). By (i) we see that $\det J_{\chi,q}(k)$ is independent of the choice of the generator $\chi$. Hence, in the remaining part of the proof, we let $\chi=\chi_{\p}$ be the Teich\"{u}muller character of $\p$ (note that $\chi_{\p}$ is a generator of $\widehat{\mathbb{F}_q^{\times}}$). 

By Lemma \ref{Lem. the Lerch permutation} we first have 
\begin{equation}\label{Eq. transformation of Jq(k) in the proof of Thm. 1}
	\det J_{\chi,q}(k)= (-1)^{\frac{(m-1)(m-2)}{2}}\det\left[J_q(\chi_{\p}^{ki},\chi_{\p}^{-kj})\right]_{1\le i,j\le m-1}.
\end{equation}
For any $1\le i<j\le m-1$, by (\ref{Eq. congruence involving Jacobi sums}) and Lemma \ref{Lem. Lucas congruence} we have 
\begin{align*}
	J_q(\chi_{\p}^{ki},\chi_{\p}^{-kj})
	&=J_q(\chi_{\p}^{-(q-1-ki)},\chi_{\p}^{-kj})\\
	&\equiv -\binom{q-1+kj-ki}{kj}\\
	&\equiv -\binom{kj-ki-1}{kj}\\
	&\equiv 0\pmod {\p}.
\end{align*}
Combining this with (\ref{Eq. transformation of Jq(k) in the proof of Thm. 1}) and noting that $J_q(\chi_{\p}^{ki},\chi_{\p}^{-ki})=(-1)^{ki+1}$ for $1\le i\le q-2$, 
we obtain 
\begin{align*}
	\det J_{\chi,q}(k)
	&\equiv (-1)^{\frac{(m-1)(m-2)}{2}}\det\left[J_q(\chi_{\p}^{ki},\chi_{\p}^{-kj})\right]_{1\le i,j\le m-1}\\
	&\equiv (-1)^{\frac{(m-1)(m-2)}{2}}\prod_{1\le i\le m-1}J_q\left(\chi_{\p}^{ki},\chi_{\p}^{-ki}\right)\\
	&\equiv (-1)^{\frac{(m-1)(m-2)}{2}}\prod_{1\le i\le m-1}(-1)^{ki+1}\\
	&\equiv (-1)^{\frac{(k+1)m^2-(k+1)m}{2}}\pmod{\p}.
\end{align*}
Since $\det J_{\chi,q}(k)\in\mathbb{Z}$, we finally obtain 
$$\det J_{\chi,q}(k)\equiv (-1)^{\frac{(k+1)m^2-(k+1)m}{2}} \pmod p.$$

In view of the above, we have completed the proof of Theorem \ref{Thm. general Jq(k)}.\qed

\section{Proof of Theorem \ref{Thm. Jq(1)}}

In this section, we set $\mathbb{F}_q=\{a_0,a_1,a_2,\cdots,a_{q-1}\}$ with $a_0=0$ and $a_1=1$. Also, we define 
\begin{equation}\label{Eq. definition of Delta q}
	\Delta_q=\prod_{1\le i<j\le q-1}\left(\chi(a_j)-\chi(a_i)\right). 
\end{equation}
We begin with the following result.

\begin{lemma}\label{Lem. explicit value of Delta q}
	Let notations be as above. Then 
	$$\Delta_q^2=(-1)^{\frac{(q-1)(q-2)+2q}{2}}(q-1)^{q-1}.$$
\end{lemma}

\begin{proof}
	As $\chi$ is a generator of $\widehat{\mathbb{F}_q^{\times}}$, we have 
	\begin{equation}\label{Eq. definition of G(t)}
		G(t):=t^{q-1}-1=\prod_{1\le j\le q-1}\left(t-\chi(a_j)\right).
	\end{equation}
	By (\ref{Eq. definition of G(t)}) we have 
	\begin{equation}\label{Eq. product of all aj}
		\prod_{1\le j\le q-1}\chi(a_j)=(-1)^q.
	\end{equation}
	Let $G'(t)=(q-1)t^{q-2}$ be the derivative of $G(t)$. Then by (\ref{Eq. definition of G(t)}) and (\ref{Eq. product of all aj}) 
	\begin{align*}
		\Delta_q^2
		&=(-1)^{\frac{(q-1)(q-2)}{2}}\prod_{1\le i\neq j\le q-1}\left(\chi(a_j)-\chi(a_i)\right)\\
		&=(-1)^{\frac{(q-1)(q-2)}{2}}\prod_{1\le j\le q-1}\prod_{i\neq j}\left(\chi(a_j)-\chi(a_i)\right)\\
		&=(-1)^{\frac{(q-1)(q-2)}{2}}\prod_{1\le j\le q-1}G'\left(\chi(a_j)\right)\\
		&=(-1)^{\frac{(q-1)(q-2)}{2}}(q-1)^{q-1}\prod_{1\le j\le q-1}\chi(a_j)^{-1}\\
		&=(-1)^{\frac{(q-1)(q-2)+2q}{2}}(q-1)^{q-1}.
	\end{align*}
	This completes the proof.
\end{proof}

We also need the following result.

\begin{lemma}\label{Lem. definition of Dy}
	Let notations be as above. Then for any $y\in\mathbb{F}_q^{\times}$ we have 
	$$D_y:=\prod_{x\in\mathbb{F}_q^{\times}\setminus\{y\}}\left(\chi(y)-\chi(x)\right)=(q-1)\chi(y)^{-1}.$$
\end{lemma}

\begin{proof}
	By (\ref{Eq. definition of G(t)}) we have 
	$$D_y=G'\left(\chi(y)\right)=(q-1)\chi(y)^{-1}.$$
	This completes the proof. 
\end{proof}

The next lemma concerns the sign of a permutation of $a_2,a_3,\cdots,a_{q-1}$.

\begin{lemma}\label{Lem. permutations induced by -x+1}
	Let $\pi_q$ be the permutation of the sequence $a_2,a_3,\cdots,a_{q-1}$ induced by the map $x\mapsto -x+1$. Then the sign of $\pi_q$ is equal to 
	$$\sign(\pi_q)=(-1)^{\frac{(q-2)(q-3)}{2}}.$$
\end{lemma}

\begin{proof}
	We first consider the case $2\mid q$. In this case, clearly $\pi_q$ can be written as a product of $(q-2)/2$ transpositions. Hence we have 
	$$\sign(\pi_q)=(-1)^{\frac{q-2}{2}}=(-1)^{\frac{(q-2)(q-3)}{2}}.$$
	
	Suppose now $2\nmid q$. Then $\{1,-1\}$ can be viewed as a subset of $\mathbb{F}_q$. Hence 
	$$\sign(\pi_q)=\prod_{2\le i<j\le q-1}\frac{(1-a_j)-(1-a_i)}{a_j-a_i}=(-1)^{\frac{(q-2)(q-3)}{2}}.$$
	This completes the proof.
\end{proof}

Now we are in a position to prove our second theorem.

{\noindent{\bf Proof of Theorem \ref{Thm. Jq(1)}.}} For any $1\le i,j\le q-2$, we clearly have 
$$J_q(\chi^i,\chi^j)=\sum_{2\le k\le q-1}\chi^i(a_k)\chi^j(1-a_k).$$
By this we obtain the matrix decomposition 
\begin{equation}\label{Eq. matrix decomposition of Jq(1)}
	J_q(1)=M_qN_q,
\end{equation}
where $M_q$ is a $(q-2)\times(q-2)$ matrix defined by 
\begin{equation*}
	M_q:=\left[\chi^i(a_k)\right]_{1\le i\le q-2,\ 2\le k\le q-1},
\end{equation*}
and $N_q$ is also a $(q-2)\times(q-2)$ matrix defined by 
\begin{equation*}
	N_q:=\left[\chi^j(1-a_k)\right]_{2\le k\le q-1,\ 1\le j\le q-2}.
\end{equation*}

We first consider $\det M_q$. Recall that $\Delta_q$ and $D_y$ are defined by (\ref{Eq. definition of Delta q}) and Lemma \ref{Lem. definition of Dy} respectively. By definition and (\ref{Eq. product of all aj}) we have 
\begin{align}
	\det M_q 
	&=\left|
	\begin{array}{cccc}
		\chi^1 (a_2)& \chi^1 (a_3)&\cdots  & \chi^1 (a_{q-1}) \\
		\chi^2 (a_2)& \chi^2 (a_3)&\cdots  & \chi^2 (a_{q-1}) \\
		\vdots  & \vdots  & \ddots &\vdots \\
		\chi^{q-2} (a_2)& \chi^{q-2} (a_3)&\cdots  & \chi^{q-2} (a_{q-1})  \\
	\end{array}
	\right| \notag \\
	&=\prod_{2\le j\le q-1}\chi(a_j)\prod_{2\le i<j\le q-1}\left(\chi(a_j)-\chi(a_i)\right) \notag \\
	&=\frac{(-1)^q\Delta_q}{D_1} \notag\\
	&=\frac{(-1)^q\Delta_q}{q-1}. \label{Eq. explicit value of det Mq}
\end{align}

We now turn to $\det N_q$. By Lemma \ref{Lem. permutations induced by -x+1} we have 
\begin{align*}
	\det N_q
	&=\left|
	\begin{array}{cccc}
		\chi^1 (1-a_2)& \chi^2 (1-a_2)&\cdots  & \chi^{q-2} (1-a_2) \\
		\chi^1 (1-a_3)& \chi^2 (1-a_3)&\cdots  & \chi^{q-2} (1-a_3) \\
		\vdots  & \vdots  & \ddots &\vdots \\
		\chi^1 (1-a_{q-1})& \chi^2 (1-a_{q-1})&\cdots  & \chi^{q-2} (1-a_{q-1})  \\
	\end{array}
	\right|\\
	&=\sign(\pi_q)\cdot 
	\left|
	\begin{array}{cccc}
		\chi^1 (a_2)& \chi^2 (a_2)&\cdots  & \chi^{q-2} (a_2) \\
		\chi^1 (a_3)& \chi^2 (a_3)&\cdots  & \chi^{q-2} (a_3) \\
		\vdots  & \vdots  & \ddots &\vdots \\
		\chi^1 (a_{q-1})& \chi^2 (a_{q-1})&\cdots  & \chi^{q-2} (a_{q-1})  \\
	\end{array}
	\right|.
\end{align*}
Now similar to the calculation of $\det M_q$, one can verify that 
\begin{equation}\label{Eq. explicit value of det Nq}
	\det N_q=(-1)^{\frac{(q-2)(q-3)+2q}{2}}\Delta_q/(q-1).
\end{equation}

Combining (\ref{Eq. explicit value of det Mq}) with (\ref{Eq. explicit value of det Nq}) and by Lemma \ref{Lem. explicit value of Delta q}, we obtain 
$$\det J_q(1)=\det M_q \cdot \det N_q=(q-1)^{q-3}.$$

In view of the above, we have completed the proof of Theorem \ref{Thm. Jq(1)}.\qed 

\section{Proof of Theorem \ref{Thm. Jq(2)}}

Let $M$ be an $r\times n$ complex matrix with $r\le n$ and let $N$ be an $n\times r$ complex matrix. Set
$$\mathcal{S}_r=\{\s=(j_1,j_2,\cdots,j_r):\ 1\le j_1<j_2<\cdots<j_r\le n\}.$$
For any $\s\in \mathcal{S}_r$, we define $M_{\s}$ (respectively $N^{\s}$) to be the $r\times r$ submatrix of $M$ (respectively submatrix of $N$) obtained by deleting all columns (respectively all rows) except those with indices in $\s$. We begin with the well-known Cauchy-Binet formula.

\begin{lemma}\label{Lem. Cauchy-Binet formula} Let $M,N$ be two complex matrices of sizes $r\times n$ and $n\times r$ respectively with $r\le n$. Then
	$$\det(MN)=\sum_{\s\in \mathcal{S}_r}\det(M_{\s})\det(N^{\s}).$$
\end{lemma}

In the remaining part of this section, we let $n=(p-1)/2$. Then it is clear that $1^2,2^2,\cdots,n^2$ are exactly all nonzero squares of $\mathbb{F}_p$. Also, we define 
\begin{equation}\label{Eq. definition of Sq}
	S_p:=\prod_{1\le i<j\le n}\left(\chi(j^2)-\chi(i^2)\right).
\end{equation}

The next result determines the explicit value of $S_p^2$. 

\begin{lemma}\label{Lem. explicit value of Sq}
	Let notations be as above. Then 
	$$S_p^2=(-1)^{\frac{n^2+n+2}{2}}n^n.$$
\end{lemma}

\begin{proof}
	Since $\chi$ is a generator of $\widehat{\mathbb{F}_p^{\times}}$, we have 
	\begin{equation}\label{Eq. definition of H(t)}
		H(t):=t^n-1=\prod_{1\le j\le n}\left(t-\chi(j^2)\right).
	\end{equation}
	By (\ref{Eq. definition of H(t)}) we obtain 
	\begin{equation}\label{Eq. product of squares of Fq}
		\prod_{1\le j\le n}\chi(j^2)=(-1)^{n+1}.
	\end{equation}
	
	Now by (\ref{Eq. definition of H(t)}) and (\ref{Eq. product of squares of Fq}) one can verify that 
	\begin{align*}
		S_p^2
		&=(-1)^{\frac{n(n-1)}{2}}\prod_{1\le i\neq j\le n}\left(\chi(j^2)-\chi(i^2)\right)\\
		&=(-1)^{\frac{n(n-1)}{2}}\prod_{1\le j\le n}\prod_{i\neq j}H'\left(\chi(j^2)\right)\\
		&=(-1)^{\frac{n(n-1)}{2}}n^n\prod_{1\le j\le n}\chi(j^2)^{-1}\\
		&=(-1)^{\frac{n^2+n+2}{2}}n^n
	\end{align*}
	This completes the proof.
\end{proof}

Let $t_1,t_2,\cdots,t_n$ be variables. For any positive integer $k$, we let 
$$P_k(t_1,t_2,\cdots,t_n)=t_1^k+t_2^k+\cdots+t_n^k.$$
Also, the $k$-th elementary symmetric polynomial of $t_1,\cdots,t_n$ is denoted by 
$e_k(t_1,\cdots,t_n)$, i.e., 
$$e_k(t_1,\cdots,t_n)=\sum_{1\le j_1<\cdots<j_k\le n}\prod_{r=1}^kt_{j_r}.$$
Also, we let $e_0(t_1,\cdots,t_n)=1$. 

We need the following result. 

\begin{lemma}\label{Lem. sums involving kth powers of squares}
	Let $\chi(i^2)=x_i$ for $1\le i\le n$. Then for any $1\le k\le n-1$, 
	$$P_k(x_1,x_2,\cdots,x_n)=x_1^k+x_2^k+\cdots+x_n^k=0$$
	
\end{lemma}

\begin{proof}
	By (\ref{Eq. definition of H(t)}) it is easy to see that 
	$$e_k:=e_k(x_1,x_2,\cdots,x_n)=0$$
	for any $1\le k\le n-1$. By the Newton identities, $P_k(x_1,\cdots,x_n)$ can be written as a polynomial of $e_1,e_2,\cdots,e_k$. This implies that 
	$$P_k(x_1,x_2,\cdots,x_n)=x_1^k+x_2^k+\cdots+x_n^k=0$$
	for any $1\le k\le n-1$. 
	This completes the proof.
\end{proof}

Now we are in a position to prove our last theorem.

{\noindent{\bf Proof of Theorem \ref{Thm. Jq(2)}.}} Note first that for any $1\le i,j\le n-1$ we have 
\begin{align*}
	J_q(\chi^{2i},\chi^{2j})
	&=\sum_{k=1}^{p-1}\chi^i(k^2)\chi^j((1-k)^2)\\
	&=\sum_{k=1}^n\chi^{i}(k^2)\left(\chi^j((1-k)^2)+\chi^j((1+k)^2)\right).
\end{align*}
By this one can verify the following matrix decomposition
\begin{equation}\label{Eq. matrix decomposition of Jq(2)}
	J_p(2)=AB,
\end{equation}
where $A$ is an $(n-1)\times n$ matrix defined by 
$$A:=\left[\chi^i(k^2)\right]_{1\le i\le n-1,\ 1\le k\le n},$$
and $B$ is an $n\times (n-1)$ matrix defined by 
$$B:=\left[\chi^j((1-k)^2)+\chi^j((1+k)^2)\right]_{1\le k\le n,\ 1\le j\le n-1}.$$
Also, for any $1\le k\le n$, we let $A_{(k)}$ be the submatrix of $A$ obtained by deleting the $k$-th column of $A$ and let $B^{(k)}$ be the submatrix of $B$ obtained by deleting the $k$-th row of $B$. By Lemma \ref{Lem. Cauchy-Binet formula} we have 
\begin{equation}\label{Eq. app of CB formula in the proof of Jq(2)}
	\det J_p(2)=\sum_{k=1}^n\det A_{(k)}\det B^{(k)}.
\end{equation}

We first consider $\det A_{(k)}$. It is clear that  
\begin{align*}
	\det A_{(k)}
	&=\left|
	\begin{array}{ccccccc}
		\chi^1 (1^2)& \cdots &\chi^1((k-1)^2)  & \chi^1 ((k+1)^2) & \cdots &\chi^1(n^2) \\
		\chi^2 (1^2)& \cdots &\chi^2((k-1)^2)  & \chi^2 ((k+1)^2) & \cdots &\chi^2(n^2) \\
		\vdots  & \vdots  & \vdots &\vdots & \ddots & \vdots \\
		\chi^{n-1} (1^2)& \cdots &\chi^{n-1}((k-1)^2)  & \chi^{n-1} ((k+1)^2) & \cdots &\chi^{n-1}(n^2) \\
	\end{array}
	\right|.
\end{align*}
Combining this with (\ref{Eq. definition of H(t)}) and (\ref{Eq. product of squares of Fq}), one can verify that 
\begin{align}
	\det A_{(k)}
	&= \prod_{i\neq k}\chi(i^2)\cdot\frac{\prod_{1\le i<j\le n}\left(\chi(j^2)-\chi(i^2)\right)}{\prod_{1\le i\le k-1}\left(\chi(k^2)-\chi(i^2)\right)\prod_{k+1\le i\le n}\left(\chi(i^2)-\chi(k^2)\right)} \notag \\
	&=\frac{S_p\cdot \prod_{1\le i\le n}\chi(i^2)}{\chi(k^2)(-1)^{n-k}\prod_{i\neq k}\left(\chi(k^2)-\chi(i^2)\right)} \notag \\
	&=\frac{(-1)^{k+1}}{n}S_p. \label{Eq. explicit value of det A(k)}
\end{align}
The last equality follows from 
$$\prod_{i\neq k}\left(\chi(k^2)-\chi(i^2)\right)=H'\left(\chi(k^2)\right)=n\chi(k^2)^{-1}.$$

Applying (\ref{Eq. explicit value of det A(k)}) to (\ref{Eq. app of CB formula in the proof of Jq(2)}) and by the determinant expansion formula, we obtain 
\begin{equation}\label{Eq. the reason why we consider B tilde}
	\det J_p(2)=\frac{S_p}{n}\sum_{k=1}^n(-1)^{k+1}\det B^{(k)}=\frac{S_p}{n}\det\widetilde{B},
\end{equation}
where $\widetilde{B}$ is an $n\times n$ matrix obtained by adding an $n\times 1$ column vector with all elements $1$ before the first column of $B$.

We now turn to $\det \widetilde{B}$. For simplicity, we let $\chi(i^2)=x_i$ for $0\le i\le p-1$. Note that $x_0=0,x_1=1$ and $x_i=x_{p-i}$ for any $1\le i\le p-1$. Using these notations we have 
\begin{equation}\label{Eq. step 1 to det tilde B}
	\det \widetilde{B}=
	\left| \begin{array}{ccccc}
		1 & x_0+x_2 & x_0^2+x_2^2 &\cdots &   x_0^{n-1}+x_2^{n-1} \\
		1 & x_1+x_3 & x_1^2+x_3^2 &\cdots &   x_1^{n-1}+x_3^{n-1} \\
		1 & x_2+x_4 & x_2^2+x_4^2 &\cdots &   x_2^{n-1}+x_4^{n-1} \\
		\vdots & \vdots & \vdots & \ddots & \vdots \\
		1 & x_{n-1}+x_{n+1} & x_{n-1}^2+x_{n+1}^2 &\cdots &   x_{n-1}^{n-1}+x_{n+1}^{n-1} 
	\end{array} \right|.
\end{equation}
By Lemma \ref{Lem. sums involving kth powers of squares} and noting that $x_n=x_{n+1}$, we have 
\begin{equation}\label{Eq. step 2 to det tilde B}
	\sum_{r=0}^{n-1}\left(x_r^k+x_{r+2}^k\right)=-x_1^k+2\sum_{r=1}^{n}x_r^k=-x_1^k
\end{equation}
for any $1\le k\le n-1$. 

In (\ref{Eq. step 1 to det tilde B}), adding the $2$nd row, $3$rd row, $\cdots$, $n$-th row to the $1$st row and by (\ref{Eq. step 2 to det tilde B}), we obtain 

\begin{equation}\label{Eq. step 3 to det tilde B}
	\det \widetilde{B}
	=\left| \begin{array}{ccccc}
		n & -x_1    & -x_1^2      &\cdots &   -x_1^{n-1} \\
		1 & x_1+x_3 & x_1^2+x_3^2 &\cdots &   x_1^{n-1}+x_3^{n-1} \\
		1 & x_2+x_4 & x_2^2+x_4^2 &\cdots &   x_2^{n-1}+x_4^{n-1} \\
		\vdots & \vdots & \vdots  & \ddots & \vdots \\
		1 & x_{n-1}+x_{n+1} & x_{n-1}^2+x_{n+1}^2 &\cdots &   x_{n-1}^{n-1}+x_{n+1}^{n-1} 
	\end{array} \right|.
\end{equation}

In (\ref{Eq. step 3 to det tilde B}), by adding the $1$st row to the $2$nd row we have 
\begin{equation}\label{Eq. step 4 to det tilde B}
	\det \widetilde{B}
	=-\frac{1}{2}\left| \begin{array}{ccccc}
		-2n & 2x_1    & 2x_1^2      &\cdots &   2x_1^{n-1} \\
		n+1 & x_3     & x_3^2       &\cdots &   x_3^{n-1} \\
		1   & x_2+x_4 & x_2^2+x_4^2 &\cdots &   x_2^{n-1}+x_4^{n-1} \\
		\vdots & \vdots & \vdots  & \ddots & \vdots \\
		1   & x_{n-1}+x_{n+1} & x_{n-1}^2+x_{n+1}^2 &\cdots &  x_{n-1}^{n-1}+x_{n+1}^{n-1} 
	\end{array} \right|.
\end{equation}
By Lemma \ref{Lem. sums involving kth powers of squares} again, for $1\le k\le n-1$ we have  
\begin{equation}\label{Eq. step 5 to det tilde B}
	2x_1^k+x_3^k+\sum_{r=2}^{n-1}\left(x_r^k+x_{r+2}^k\right)=-x_2^k+2\sum_{r=1}^{n}x_r^k=-x_2^k.
\end{equation}

In (\ref{Eq. step 4 to det tilde B}), adding the $2$nd row, $3$rd row, $\cdots$, $n$th row to the $1$st row and by (\ref{Eq. step 5 to det tilde B}), we have 
\begin{equation}\label{Eq. step 6 to det tilde B}
	\det \widetilde{B}
	=\frac{1}{2}\left| \begin{array}{ccccc}
		1 & x_2    & x_2^2      &\cdots &   x_2^{n-1} \\
		n+1 & x_3     & x_3^2       &\cdots &   x_3^{n-1} \\
		1   & x_2+x_4 & x_2^2+x_4^2 &\cdots &   x_2^{n-1}+x_4^{n-1} \\
		\vdots & \vdots & \vdots  & \ddots & \vdots \\
		1   & x_{n-1}+x_{n+1} & x_{n-1}^2+x_{n+1}^2 &\cdots &  x_{n-1}^{n-1}+x_{n+1}^{n-1} 
	\end{array} \right|.
\end{equation}

In (\ref{Eq. step 6 to det tilde B}), by using the row operations recursively, one can verify that 
\begin{equation}\label{Eq. step 7 to det tilde B}
	\det \widetilde{B}
	=\frac{1}{2}\left| \begin{array}{ccccc}
		1          & x_2    & x_2^2      &\cdots &   x_2^{n-1} \\
		n+1        & x_3     & x_3^2       &\cdots &   x_3^{n-1} \\
		\alpha_4   & x_4 & x_4^2 &\cdots &   x_4^{n-1} \\
		\vdots & \vdots & \vdots  & \ddots & \vdots \\
		\alpha_{n}   & x_{n} & x_{n}^2 &\cdots &  x_{n}^{n-1}\\
		\alpha_{n+1}   & x_{n+1} & x_{n+1}^2 &\cdots &  x_{n+1}^{n-1} 
	\end{array} \right|,
\end{equation}
where 
\begin{equation}\label{Eq. definition of alpha k in the proof of Thm. 3}
	\alpha_k=\begin{cases}
		\frac{1-(-1)^{k/2}}{2}  & \mbox{if}\ k\equiv 0\pmod 2,\\
		(-1)^{(k+1)/2}\cdot(n+1)+\frac{1+(-1)^{(k-1)/2}}{2} & \mbox{otherwise.}
	\end{cases}
\end{equation}

In (\ref{Eq. step 7 to det tilde B}), by subtracting the $(n-1)$-th row from the $n$-th row and noting that $x_n=x_{n+1}$, we obtain 

\begin{align}\label{Eq. step 8 to det tilde B}
	\det \widetilde{B}
	&=\frac{1}{2}\left| \begin{array}{ccccc}
		1          & x_2    & x_2^2      &\cdots &   x_2^{n-1} \\
		n+1        & x_3     & x_3^2       &\cdots &   x_3^{n-1} \\
		\alpha_4   & x_4 & x_4^2 &\cdots &   x_4^{n-1} \\
		\vdots & \vdots & \vdots  & \ddots & \vdots \\
		\alpha_{n}   & x_{n} & x_{n}^2 &\cdots &  x_{n}^{n-1}\\
		\alpha_{n+1}-\alpha_n  &0 & 0 &\cdots &  0
	\end{array} \right| \notag \\
	&=\frac{1}{2}\cdot (-1)^{n+1} \cdot \beta_p\cdot 
	\left|\begin{array}{ccccc}
		x_2 & x_2^2 &\cdots & x_2^{n-1} \\
		x_3 & x_3^2 &\cdots &   x_3^{n-1} \\
		x_4 & x_4^2 &\cdots &   x_4^{n-1} \\
		\vdots & \vdots & \ddots & \vdots \\
		x_n & x_n^2 &\cdots &   x_n^{n-1} \\
	\end{array}\right|,
\end{align}
where $\beta_p=\alpha_{n+1}-\alpha_{n}$ and by (\ref{Eq. definition of alpha k in the proof of Thm. 3}) one can verify that 
\begin{equation}\label{Eq. explicit value of beta p in the proof of Thm. 3}
	\beta_p=(-1)^{\frac{n(n-1)}{2}+n+1}\cdot\frac{p+(-1)^{n+1}}{2}. 
\end{equation}

Now by (\ref{Eq. step 8 to det tilde B}) and (\ref{Eq. product of squares of Fq}) we finally obtain 
\begin{align}
	\det \widetilde{B}
	&=\frac{1}{2}\cdot (-1)^{n+1} \cdot \beta_p\left(\prod_{2\le i\le n}x_i\right)\prod_{2\le i<j\le n}\left(x_j-x_i\right) \notag \\
	&=\frac{1}{2}\cdot \beta_p\cdot \frac{S_p}{\prod_{2\le i\le n}\left(x_i-x_1\right)} \notag \\
	&=\frac{(-1)^{n-1}}{2n}\beta_pS_p. \label{Eq. explicit value of DET tilde B}
\end{align}
The last equality follows from 
$$\prod_{2\le i\le n}\left(x_i-x_1\right)=(-1)^{n-1}H'(x_1)=(-1)^{n-1}n.$$

By (\ref{Eq. the reason why we consider B tilde}), (\ref{Eq. explicit value of beta p in the proof of Thm. 3}), (\ref{Eq. explicit value of DET tilde B}) and Lemma \ref{Lem. explicit value of Sq}, we finally obtain 
\begin{equation*}
	\det J_p(2)=\frac{S_p}{n}\det\widetilde{B}=\frac{1+(-1)^{\frac{p+1}{2}}p}{4}\left(\frac{p-1}{2}\right)^{\frac{p-5}{2}}.
\end{equation*}

In view of above, we have completed the proof of Theorem \ref{Thm. Jq(2)}.\qed

\begin{remark}\label{Remark. Why we cannot generalize Thm. 3}
	Now we explain why our method used in the above proof cannot be applied to $\mathbb{F}_{p^l}$ with $l\ge2$. In fact, in the prime field $\mathbb{F}_p$, we can naturally list all nonzero squares as $1^2,2^2,\cdots,n^2$. Let $x_i=\chi(i^2)$ for $0\le i\le q-1$. Then it is easy to verify that the $(i,j)$-entry of the matrix $B$ defined by (\ref{Eq. matrix decomposition of Jq(2)}) can be simplified as $x_{i-1}^j+x_{i+1}^j$. 
	
	However, for $\mathbb{F}_{p^l}$ with $l\ge2$, if we choose elements $a_1,\cdots,a_{(p^l-1)/2}$ such that 
	$$a_1^2,\cdots,a_{(p^l-1)/2}^2$$
	are exactly all the nonzero squares of $\mathbb{F}_{p^l}$. Then we can easily get the matrix decomposition 
	$$J_q(2)=A'B',$$
	where 
	$$A'=\left[\chi^i(a_k^2)\right]_{1\le i\le \frac{p^l-3}{2},\ 1\le k\le \frac{p^l-1}{2}},$$
	and 
	$$B'=\left[\chi^j\left((1-a_k)^2\right)+\chi^j\left((1+a_k)^2\right)\right]_{1\le k\le \frac{p^l-1}{2},\ 1\le j\le \frac{p^l-3}{2}}.$$
	Similar to (\ref{Eq. app of CB formula in the proof of Jq(2)}), we also have 
	$$\det J_q(2)=\sum_{k=1}^{\frac{p^l-1}{2}}\det A'_{(k)}\det B'^{(k)},$$ 
	and it is easy to verify that 
	$$\det A'_{(k)}=(-1)^{k+1}\frac{2}{p^l-1}S_q,$$
	where 
	$$S_q=\prod_{1\le i<j\le (p^l-1)/2}\left(\chi(a_j^2)-\chi(a_i^2)\right).$$
	Using this we obtain 
	$$\det J_q(2)=\frac{2}{p^l-1}S_q\det\widetilde{B'},$$
	where $\widetilde{B}$ is an $(p^l-1)/2\times (p^l-1)/2$ matrix obtained by adding an $(p^l-1)/2\times 1$ column vector with all elements $1$ before the first column of $B'$.
	
	Now our main obstacle is the calculation of $\det\widetilde{B'}$. This is because we cannot represent the $(i,j)$-entry of $B'$ in a simple form like that of $B$. 
\end{remark}

\section{Concluding Remarks}

Greene \cite[Definition 2.4]{Greene} posed an analogue of binomial coefficients. For any $A,B\in\widehat{\mathbb{F}_q^{\times}}$, Greene defined 
$$\binom{A}{B}:=\frac{B(-1)}{q}J(A,\overline{B}),$$
where $\overline{B}\in\widehat{\mathbb{F}_q^{\times}}$ such that $\overline{B}(x)=B(x)^{-1}$ for any $x\in\mathbb{F}_q^{\times}$. Motivated by Greene's analogue and our theorems, investigating
$$\det\left[\binom{\chi^{i+j}}{\chi^i}\right]_{1\le i,j\le q-2}\ \text{and}\ 
\det \left[\binom{\chi^{2i+2j}}{\chi^{2i}}\right]_{1\le i,j\le (q-3)/2}$$
might be also meaningful. However, we cannot solve this problem currently.

\section{Appendix}

In this appendix we prove (\ref{Eq. det Cn}). It is known that for any positive integers $r,l$, we have 
\begin{equation}\label{Eq. Appendix A}
	\binom{r}{l}=\binom{r-1}{l-1}+\binom{r-1}{l}.
\end{equation}

Now we calculate $\det C_n=\det \left[\binom{i+j}{i}\right]_{1\le i,j\le n}$. The case $n=1$ is trivial. Assume now $n\ge2$. Subtracting the $(n-1-i)$-th row from the $(n-i)$-th row for $i=0,1,\cdots,n-2$ sequentially and by (\ref{Eq. Appendix A}), one can verify that 
\begin{equation}\label{Eq. Appendix B}
	\det C_n=\det \left[\binom{i+j}{i}\right]_{1\le i,j\le n}=
	\left|\begin{array}{ccccc}
		2 &            3 & \cdots & n+1\\
		\binom{2}{2} & \binom{3}{2} & \cdots & \binom{n+1}{2}\\
		\vdots       & \vdots       & \ddots & \vdots \\
		\binom{n}{n} & \binom{n+1}{n} & \cdots & \binom{2n-1}{n}
	\end{array}
	\right|.
\end{equation}
We next consider the right hand side of (\ref{Eq. Appendix B}). Subtracting the $(n-1-i)$-th column from the $(n-i)$-th column for $i=0,1,\cdots,n-2$ sequentially and by (\ref{Eq. Appendix A}) again, we obtain 
\begin{equation}\label{Eq. Appendix C}
	\det C_n=
	\left|\begin{array}{ccccc}
		1+1 &            1+0   & \cdots & 1+0               & 1+0 \\
		1 &   \binom{2}{1}   & \cdots & \binom{n-1}{1}    & \binom{n}{1} \\
		\vdots &   \vdots         & \ddots & \vdots            & \vdots \\
		1 & \binom{n-1}{n-2} & \cdots & \binom{2n-4}{n-2} & \binom{2n-3}{n-2} \\
		1 & \binom{n}{n-1}   & \cdots & \binom{2n-3}{n-1} & \binom{2n-2}{n-1}
	\end{array}
	\right|=\det D_n+\det C_{n-1}, 
\end{equation}
where 
$$\det D_n=\left|\begin{array}{ccccc}
	1      &            1     & \cdots & 1                 & 1 \\
	1      &  \binom{2}{1}    & \cdots & \binom{n-1}{1}    & \binom{n}{1} \\
	\vdots &  \vdots          & \ddots & \vdots            & \vdots \\
	1      & \binom{n-1}{n-2} & \cdots & \binom{2n-4}{n-2} & \binom{2n-3}{n-2} \\
	1      & \binom{n}{n-1}   & \cdots & \binom{2n-3}{n-1} & \binom{2n-2}{n-1}
\end{array}
\right|.$$
For $\det D_n$, subtracting the $(n-1-i)$-th row from the $(n-i)$-th row for $i=0,1,\cdots,n-2$ sequentially and using (\ref{Eq. Appendix A}), we see that 
$$\det D_n=\left|\begin{array}{ccccc}
	1      &            1     & \cdots & 1                 & 1 \\
	0      &  \binom{1}{1}    & \cdots & \binom{n-2}{1}    & \binom{n-1}{1} \\
	\vdots &  \vdots          & \ddots & \vdots            & \vdots \\
	0      & \binom{n-2}{n-2} & \cdots & \binom{2n-5}{n-2} & \binom{2n-4}{n-2} \\
	0      & \binom{n-1}{n-1} & \cdots & \binom{2n-4}{n-1} & \binom{2n-3}{n-1}
\end{array}
\right|=\left|\begin{array}{ccccc}
	1      & \cdots & \binom{n-2}{1}    & \binom{n-1}{1} \\
	\vdots & \ddots & \vdots            & \vdots \\
	1      & \cdots & \binom{2n-5}{n-2} & \binom{2n-4}{n-2} \\
	1      & \cdots & \binom{2n-4}{n-1} & \binom{2n-3}{n-1}
\end{array}
\right|.$$
Repeating the above procedure, one can easily verify that 
\begin{equation}\label{Eq. Appendix D}
	\det D_n=1.
\end{equation}
Combining (\ref{Eq. Appendix D}) with (\ref{Eq. Appendix C}) and noting that $\det C_1=2$, we finally obtain $\det C_n=n+1$. This completes the proof of (\ref{Eq. det Cn}).

\subsection*{Acknowledgements}

The authors would like to thank the referee for careful reading and helpful comments. This research was supported by the National Natural Science Foundation of China (Grant Nos. 12101321, 12201291 and 12071208). The first author was also supported by Natural Science Foundation of Nanjing University of Posts and Telecommunications (Grant No. NY224107).

\end{document}